\documentclass[a4paper,11pt]{article}
\usepackage{amssymb,amsmath,amsthm}
\usepackage{graphics,graphicx,color}
\usepackage{srcltx}
\DeclareGraphicsExtensions{.eps} \textwidth=169mm
\definecolor{darkred}{rgb}{0.9,0.1,0.1}

\newtheorem{proposition}{Proposition}
\newtheorem{theorem}{Theorem}
\newtheorem{lemma}{Lemma}

\newtheorem{remark}{Remark}
{\rm}
\definecolor{darkred}{rgb}{0.9,0.1,0.1}

\def\eps{\varepsilon}

\author{A. Piatnitski$^{\small\bf a,b}$ and E. Zhizhina$^{\small\bf a}$\\
\\
{
\small $^{\small\bf a}$Institute for Information Transmission Problems of RAS,}\\[-1.5mm]
{\small 19, Bolshoi Karetnyi per. build.1,}\\[-1.5mm]
{\small 127051 Moscow, Russia}\\
{\small $^{\small\bf b}$The Arctic University of Norway, Campus Narvik,}\\[-1.5mm]
{\small P.O.Box 385, Narvik 8505, Norway}\\
$ $}

\def\eps{\varepsilon}

\begin{document}

\title{Large time behaviour of symmetric random walk in high-contrast periodic environment\thanks{The research  has been partially supported by the Russian Science Foundation (project No. 14-50-00150)}}

\maketitle

\begin{abstract}
The paper deals with the asymptotic properties of a symmetric random walk in a high contrast periodic medium in $\mathbb Z^d$, $d\geq 1$.
We show that under proper diffusive scaling  the random walk exhibits a non-standard limit behaviour.
In addition to the coordinate of the random walk in $\mathbb Z^d$ we introduce an extra variable that characterizes the position of the random walk in the period and show that this two-component process
converges in law to a limit Markov process. The components of the limit process are mutually coupled, thus we cannot expect that the limit behaviour of the coordinate process is Markov.
We also prove the convergence in the path space for the said random walk.
\end{abstract}

\noindent

\section*{Introduction}

We study in this work the large time behaviour of a symmetric random walk in $\mathbb Z^d$, $d\geq 1$,
under the assumptions that the medium is periodic, elliptic and high-contrast. More precisely, we assume that the transition probabilities of the random walk depend on a small parameter $\varepsilon>0$ and
that they are of order one for some links on the period and of order $\varepsilon^2$ for other links. It is assumed, moreover, that the graph of links of order one forms an unbounded connected set in $\mathbb Z^d$. Denoting this random walk $\widehat X(n)$ we study the limit behaviour of the process
$\widehat X_\varepsilon(t)=\varepsilon \widehat X([t/\varepsilon^2])$, as $\varepsilon\to 0$.

Various phenomena in media with a high-contrast microstructure have been widely studied by the specialists
in applied sciences and then since '90th high-contrast homogenization problems have been attracting the attention of mathematicians.

Homogenization problems for partial differential equations describing high-contrast periodic media have
been widely investigated in the existing mathematical literature. In the pioneer work \cite{ADH90} a parabolic equation with high-contrast periodic coefficients has been considered. It was shown that the effective equation contains a non-local in time term which represents the memory effect. In the literature on porous media these models are usually called
double porosity models. Later on in \cite{Ale92}, 
with the help of two-scale convergence techniques, it was proved that the solutions of the original parabolic equations two-scale converge to a function which depends both on slow and fast variables, and, as a function of fast and slow variables, satisfies a system of local PDEs.

In the case of spectral problems the homogenized spectral problem turns out to be non-linear with respect to the spectral parameter. The convergence of spectra and the structure of the limit operator pencils
have been considered in \cite{Zhi00}, \cite{BKS08} and other works.

A number of works have been devoted to nonlinear double porosity models, see \cite{BLM}, \cite{BCP04}
and references therein. In particular, for the evolution nonlinear models the memory effect was also observed.

In the discrete setting homogenization problems for high-contrast equations and Lagrangians were studied in \cite{BCP15}.
For evolution high-contrast difference equations the two-scale limit of solutions is a function
of continuous "slow" \ variables and discrete "fast" \ variables.

The appearance of a non-local term in the homogenized equation means that the limit in law of
the scaled random walks need not be a Markov process. Our goal is to study the large time behaviour
of the random walk $\widehat X(n)$. It turns out that in order to keep the Markovity of the limit process
one can equip the coordinate process $X(n)$ with an additional variable, $k(\widehat X(n))$, that specifies
the position of the random walk in the period. Although in the original process $(\widehat X(n), k(\widehat X(n)))$
the last component is a function of $\widehat X(n)$, in the limit process the last component is independent
of the other component.

The limit process  is a two-component continuous time Markov process $\mathcal{X}(t) = (\widehat {\mathcal{X}}(t), k(t))$, its  first component $\widehat{\mathcal{X}}(t)$  lives in the space $\mathbb R^d$, while the second component is a jump Markov process $k(t)$ with a finite number of states $k(t) \in \{0,1, \ldots, M \}$. The process $k(t)$
does not depend on $\widehat{\mathcal{X}}(t)$;  the  intensities $\lambda(k)$ and transition probabilities $\mu_{k j}, \; k\neq j, \ k,j = 0,1, \ldots, M,$  of its jumps are expressed in terms of the transition
probabilities of the
original symmetric random walk.  When $k(t) = 0$,  the first
component $\widehat{\mathcal{X}}(t)$ evolves along the trajectories of a Brownian motion in $\mathbb R^d$, but when $k(t) \neq 0$, then the first
component remains still until the second component of the process
takes again the value equal to 0. Thus the trajectories of  $\widehat{\mathcal{X}}(t)$  coincide with  the trajectories of a Brownian motion in $\mathbb R^d$ on those time intervals where  $k(t)=0$.  As long as $k(t) \neq 0$, then $\widehat{\mathcal{X}}(t)$ does
not move, and only the second component of the process evolves, that is, figuratively speaking, the process
lives during this period in the "astral" space $A= \{x_1,\ldots,  x_M \}$.

We also study the generalization of this model to the case of several fast
components. More precisely, we assume that the set of links to which transition probabilities of order one are assigned consists of a finite number of non-intersecting unbounded connected components.
In this case  we also equip the random walk with an additional variable, however it indicates not only whether the random walk is in the "astral" space or not, but also specifies the "fast" \ subset to which the random walk belongs.
  Also we associate to each fast component the corresponding effective covariance matrix. The limit two-component  process is Markov, its second coordinate is a Markov jump process with a finite number of states. When the second coordinate indicates the "astral" \ state, the first one does not move. Otherwise, the first coordinate is a diffusion in $\mathbb R^d$, however its covariance matrix depends on the value of the second coordinate.

Our approach relies on approximation results from \cite{EK}.  A crucial step here is constructing several periodic correctors which are introduced as solutions of auxiliary difference elliptic equations on the period. The coefficients of the corresponding difference operator on the discrete torus are defined as the transition probabilities of $\widehat X(n)$ with $\varepsilon=0$.
Earlier the corrector techniques in the discrete framework have been developed in \cite{Ko86} for proving the homogenization results for uniformly elliptic difference schemes.

We prove the convergence, as $\varepsilon\to0$, of semigroups generated by $(\widehat X_\varepsilon(t),\,k(\widehat X_\varepsilon(t)))$ and determine the generator of the
limit semigroup. This yields the convergence of finite dimensional distributions of $(\widehat X_\varepsilon(t),\,k(\widehat X_\varepsilon(t)))$.
We then improve this result and show that $(\widehat X_\varepsilon(t),\,k(\widehat X_\varepsilon(t)))$ converges in law in the topology
of $D[0,\infty)$.

It is interesting to observe that, unlike  diffusion models, the high-contrast discrete
models are feasible in any dimension including $d=1$, at the price of admitting not only nearest neighbour interactions. The graph of non-vanishing transition probabilities should be large enough
to ensure the existence of unbounded connected component.

\section{Problem setup}

\bigskip


We consider a symmetric random walk $\widehat X(n)$ on $\mathbb Z^d, \ d \ge 1$, with transition probabilities
$p(x,y)=\Pr (x \to y)$,  $(x,y) \in \mathbb Z^d \times  \mathbb Z^d$:
\begin{equation}\label{p}
p(x,y) = p(y,x), \quad (x,y) \in \mathbb Z^d \times  \mathbb Z^d; \quad \sum_{y \in \mathbb Z^d} p(x,y)=1 \quad  \forall x \in \mathbb Z^d.
\end{equation}

We assume that the random walk satisfies the following properties:
\begin{itemize}
\item [-] {\it Periodicity}. The functions $p(x,x+\xi)$ are  are periodic in $x$  with a period $Y$ for all $\xi\in\mathbb Z^d$. In what follows we identify the period $Y$ with the corresponding $d$-dimensional discrete torus $\mathbb T^d$.
\item [-] {\it Finite range of interactions}. There exists $c_1>0$ such that
\begin{equation}\label{c1}
p(x,x+\xi)=0, \quad \hbox{if }|\xi|>c_1.
\end{equation}
\item [-] {\it Irreducibility.}  The random walk is irreducible in $\mathbb Z^d$.
\end{itemize}
We denote the transition matrix of the random walk by $P =\{ p(x,y), \ x,y \in \mathbb Z^d \}$.

\medskip
In this paper we consider a family of transition probabilities $p^{(\varepsilon)}(x,y)$ that satisfy the properties formulated above and depend on a small parameter $\varepsilon>0$.
These transition probabilities describe  the so-called high-contrast periodic structure of the environment. We suppose that the transition matrix $P^{(\varepsilon)}$ is a small perturbation
of a fixed transition matrix $P^0$ and can be represented as
\begin{equation}\label{PV}
P^{(\varepsilon)} = P^0 + \varepsilon^2 V.
\end{equation}
In the sequel the upper index $(\varepsilon)$ is dropped.

In order to characterize the matrices $P^0$ and $V$ we divide the periodicity cell into
two sets
\begin{equation}\label{defAB}
\mathbb T^d = A \cup B; \quad A,\, B \neq \emptyset, \; A \cap B =   \emptyset,
\end{equation}
and assume that $B$ is a connected set such that its periodic extension denoted $B^{\sharp}$
is unbounded and connected. Here the connectedness is understood in terms of the transition matrix $P^0$ that is two points $x',\,x''\in\mathbb Z^d$ are connected if there exists a path $x^1,\ldots,x^L$ in
 $\mathbb Z^d$ such that $x^1=x'$, $x^L=x''$ and $p_0(x^j,x^{j+1})>0$ for all $j=1,\ldots, L-1$. We also denote by $A^\sharp$ the periodic extension of $A$. Then
$\mathbb Z^d = A^{\sharp} \cup B^{\sharp}$.

\bigskip
We impose the following conditions on $P^0$ and $V$:
\begin{itemize}
\item[\bf --] $P^0$ satisfies conditions (\ref{p})
\item[\bf --] $p_0(x,x) = 1$, if $x \in A^\sharp$;
%
\item[\bf --] $p_0(x,y) = 0$, if $x, y \in A^\sharp, \ x \neq y$;
\item[\bf --] $p_0(x,y) = 0$, if $x \in B^\sharp, \ y \in A^\sharp$;
%
%
\item[\bf --]  the elements of matrix $V$ satisfy the relation
\begin{equation}\label{v0}
\sum_{y \in \mathbb Z^d} v (x,y) = 0 \quad \forall x \in \mathbb Z^d,
\end{equation}
\end{itemize}
Notice that, as a consequence of the above conditions, $B^\sharp$ is a maximal connected component and, consequently, $P^0$ is irreducible on $B^\sharp$.
From the periodicity of $V$ it also follows that
$$
v_{max}:=\max_{x,y \in \mathbb Z^d} |v(x,y)| <\infty.
$$
Under these conditions, for the transition probabilities defined in \eqref{PV},
if $p(x,y) \neq 0$, then
\begin{itemize}
\item[\bf --] $p(x,y) \asymp 1$, when $x,y \in B^\sharp$ (rapid movement);
\item[\bf --] $p(x,y) \asymp\varepsilon^2$, when $x, y \in A^\sharp, \ x \neq y$ (slow movement);
\item[\bf --] $p(x,y) \asymp\varepsilon^2$, when $x \in B^\sharp, \ y \in A^\sharp$ (rare exchange between $A^\sharp$ and $B^\sharp$).
\end{itemize}

Let us notice that for $x,y \in B^\sharp$ we have $p_0(x,y) = \Pr (x \to y)|_{\varepsilon=0}= \lim\limits_{\varepsilon\to0}\Pr (x \to y| \ \mbox{no entry to} \ A^\sharp)$.
The above choice of the transition probabilities reflects a significant slowdown of the random walk inside of high-contrast periodic environments. In what follows we study the large time behavior of this random walk and use $\varepsilon$  as the corresponding scaling factor.

\medskip
We introduce now the rescaled process. Let $\varepsilon \mathbb Z^d = \{x: \frac{x}{\varepsilon} \in \mathbb Z^d \}$ be a compression of the lattice $\mathbb Z^d$, then $\varepsilon \mathbb Z^d = \varepsilon A^{\sharp} \cup \varepsilon B^{\sharp}$.
Let $l_0^{\infty} (\mathbb Z^d)$ be the Banach space of bounded functions on $\mathbb Z^d$ vanishing at infinity with the norm $\|f \| = \sup_{x \in \mathbb Z^d} |f(x)|$.
We denote by $T_{\varepsilon}$ the transition operator
\begin{equation}\label{Peps}
T_{\varepsilon} f(x) = \sum_{y \in \varepsilon \mathbb Z^d} p_\varepsilon (x,y) f(y), \quad f \in  l_0^{\infty} (\varepsilon \mathbb Z^d),
\end{equation}
where $p_\varepsilon (x,y) = p(\frac{x}{\varepsilon}, \frac{y}{\varepsilon})$, and $p(x,y)$ is defined above in (\ref{p}) - (\ref{PV}).
Then the operator
\begin{equation}\label{L_e}
L_\varepsilon \ = \ \frac{1}{\varepsilon^2} (T_{\varepsilon}-I)
\end{equation}
is the difference generator of the random walk $\widehat X_\varepsilon (t) =  \varepsilon \widehat X(\left[\frac{t}{\varepsilon^2}\right]) $ on $\varepsilon \mathbb Z^d$  with transition operator $T_\varepsilon$.

The goal of the paper is to describe the large time behavior of the random walk $\widehat X_\varepsilon (t)$ and to construct the limit
process.

\section{Semigroup convergence}
\label{sec2}

In this section we supplement the random walk $\widehat X_\varepsilon (t)$ with an additional component,
and, for the extended process, prove the convergence of the corresponding semigroups.
  Assume that the set $A$ defined in \eqref{defAB} contains $M \in \mathbb N$ sites of $\mathbb T^d$: $A = \{ x_1, \ldots, x_M \}$.
For each $k=1, \ldots, M$ we denote by $\{x_k \}^{\sharp}$ the periodic extension of the point $x_k \in A$, then
\begin{equation}\label{decZ}
\varepsilon \mathbb Z^d  = \varepsilon B^{\sharp} \cup \varepsilon A^{\sharp} = \varepsilon B^{\sharp} \cup \varepsilon \{ x_1 \}^{\sharp} \cup \ldots  \cup \varepsilon \{ x_M \}^{\sharp}.
\end{equation}
We assign to each $x \in \varepsilon \mathbb Z^d$ the index $k(x) \in \{0,1, \ldots, M\}$ depending on the component in decomposition (\ref{decZ}) to which $x$ belongs:
\begin{equation}\label{kx}
k(x) \ = \ \left\{
\begin{array}{l}
0, \; \mbox{ if } \;  x \in \varepsilon B^{\sharp}; \\
j, \;  \mbox{ if } \;  x \in \varepsilon \{ x_j \}^{\sharp}, \; j= 1, \ldots, M.
\end{array}
\right.
\end{equation}
With this construction in hands we  introduce the metric space
\begin{equation}\label{Ee}
E_\varepsilon\ = \  \left\{ (x, k(x)), \; x \in \varepsilon \mathbb Z^d,\; k(x) \in \{0,1,\ldots, M \} \right\}, \quad E_\varepsilon \subset \varepsilon \mathbb Z^d \times \{0,1,\ldots, M \}
 \end{equation}
 with a metric that coincides with the metric in  $\varepsilon \mathbb Z^d$ for the first component of $(x,k(x)) \in E_\varepsilon$.
We denote by ${\cal B}(E_\varepsilon)$ the space of bounded functions on $E_\varepsilon$ and introduce the transition operator $T_\varepsilon$ of the random walk $X_\varepsilon(t) = (\widehat X_\varepsilon (t), k(\widehat X_\varepsilon (t)))$ on $E_\varepsilon$ using the transition operator (\ref{Peps}) of the random walk on $\varepsilon \mathbb Z^d$:
\begin{equation}\label{Te}
(T_\varepsilon f)(x, k(x)) = \sum_{y \in \varepsilon \mathbb Z^d } p_\varepsilon (x,y) f(y ,k(y)), \quad f \in {\cal B}(E_\varepsilon).
\end{equation}
Then $T_\varepsilon$ is the contraction on  ${\cal B}(E_\varepsilon)$:
$$
\| T_\varepsilon f \|_{{\cal B}(E_\varepsilon)} = \sup_{(x, k(x))} |T_\varepsilon f (x, k(x))| \le \sup_{(x, k(x))} |f (x, k(x))|,  \quad f \in {\cal B}(E_\varepsilon).
$$
\begin{remark}\label{RT}
Since the point  $(x,k(x)) \in E_\varepsilon$ is uniquely defined by its first coordinate  $x \in \varepsilon \mathbb Z^d$, then we can use  $x \in \varepsilon \mathbb Z^d$ as a coordinate in $E_\varepsilon$ (considering $E_\varepsilon$ as a graph of the mapping $k: \varepsilon \mathbb Z^d \to \{0,1,\ldots, M \}$). In particular, for the transition probabilities of the random walk on $E_\varepsilon$ we keep the same notations $p_\varepsilon(x,y)$ as in \eqref{Peps}.
\end{remark}
\bigskip

We proceed to constructing the limit semigroup.   We denote $E= \mathbb R^d \times \{0 ,1, \ldots, M\}$, and $ C_0 (E)$ stands for the Banach space of continuous  functions vanishing at infinity. A function $F=F(x,k) \in C_0(E)$ can be represented as a vector function
$$
F(x,k) = \{ f_k (x) \in C_0(\mathbb R^d), \; k=0,1,\ldots, M \}.
$$
The norm in $C_0(E)$ is given by
$$
\|F\|_{C_0(E)} = \max_{k = 0, 1 \ldots, M}  \|f_k\|_{C_0(\mathbb R^d)}.
$$
Consider the operator
\begin{equation}\label{LM}
(L F)(x,k) = \left(
\begin{array}{c}
\Theta \cdot \nabla \nabla f_0 (x) \\
0  \\ \cdots \\ 0
\end{array}
\right) \ + \ L_A F(x,k),
\end{equation}
where  $\Theta$ is  a positive definite matrix defined below in \eqref{theta}, and  $L_A$ is a generator of a Markov jump process
\begin{equation}\label{LA}
L_A F(x,k) \ = \ \lambda(k)\ \sum_{{j=0}\atop{j \neq k}}^M \mu_{kj} (f_j(x) - f_k(x))
\end{equation}
with
\begin{equation}\label{alpha_prior}
\alpha_{0j} \ = \ \frac{1}{|B|} \ \sum_{y \in B} v(y,y_j),\quad
\alpha_{j0} = \sum_{y \in B} v(y_j, y), \quad
\alpha_{kj} = v(y_k, y_j), \ \ j,\,k=1,\ldots,M,\ j\not=k,
\end{equation}
\begin{equation}\label{mu_prior}
 \lambda(k) = \sum_{{j=0}\atop {j \neq k}}^M \alpha_{kj}, \quad \mu_{kj} = \frac{\alpha_{kj}}{\lambda(k)}.
\end{equation}
Observe that  
$$
0< \lambda_0 \le \min_k \lambda(k) \le \max_k \lambda(k) \le \lambda_1 < \infty, \quad \mu_{kj} \ge 0, \quad \sum_{{j=0}\atop{j \neq k}}^M \mu_{kj} =1 \; \; \forall \ k.
$$
The operator $L$ is defined on the core
\begin{equation}\label{core}
D \ = \ \{ (f_0, f_1, \ldots, f_M), \; f_0 \in C_0^{\infty}(\mathbb R^d), \; f_j \in C_0(\mathbb R^d), \; j=1, \ldots, M \} \ \subset \ C_0(E)
\end{equation}
which is a dense set in $ C_0 (E)$.
One can check that the operator $L$ on $C_0 (E)$ satisfies the positive maximum principle, i.e. if
$
F \in C_0(E)$  and  $\max_{ E } F(x,k) =  F(x_0, k_0) = f_{k_0}(x_0),
$
then  $L F (x_0, k_0) \le 0$. Indeed, from (\ref{LM}) - (\ref{LA}) we obtain
$$
L F(x_0, 0 ) = \Theta \cdot \nabla \nabla f_0 (x_0) + L_A F(x_0,0) \le 0  \;   \mbox{ in the case  } \; (x_0, k_0) = (x_0,0),
$$
and
$$
L F(x_0, k) = L_A F(x_0, k) \le 0 \;    \mbox{ in the case } \; (x_0, k_0) = (x_0,k), \; k \neq 0.
$$
Then by the Hille-Yosida theorem the closure of $L$ is a generator of a strongly continuous, positive, contraction semigroup $T(t)$  on $ C_0 (E)$, that is a Feller semigroup.
\medskip

For every $F \in C_0(E)$ we define on  $E_\varepsilon$ the function $\pi_\varepsilon F$ as follows:
\begin{equation}\label{1}
(\pi_\varepsilon  F) (x, k(x)) \ = \ \left\{
\begin{array}{ll}
f_0 (x), & \mbox{if} \; x \in \varepsilon B^{\sharp}, \; k(x) =0; \\
f_1 (x), & \mbox{if} \; x \in \varepsilon \{ x_1 \}^{\sharp}, \; k(x) =1;\\
\cdots \\
f_M (x), & \mbox{if} \; x \in \varepsilon \{ x_M \}^{\sharp}, \; k(x) =M.
\end{array}
\right.
\end{equation}
Let $l_0^{\infty}(E_\varepsilon)$ be a Banach space of functions on $E_\varepsilon$ vanishing as $|x| \to \infty$ with the norm
\begin{equation}\label{normle}
\|f\|_{l_0^{\infty}(E_\varepsilon)} = \sup_{(x,k(x)) \in E_\varepsilon}|f(x,k(x))| = \sup_{x \in \varepsilon \mathbb Z^d}|f(x,k(x))|.
\end{equation}
Then $\pi_\varepsilon$ defines a bounded linear transformation $ \pi_\varepsilon: C_0(E) \to l_0^{\infty}(E_\varepsilon)$:
\begin{equation}\label{hatpi1}
\| \pi_\varepsilon F \|_{l_0^{\infty}(E_\varepsilon)} = \sup_{ (x, k(x)) \in E_\varepsilon} |(\pi_\varepsilon F) (x, k(x))| \le \|  F \|_{C_0(E)}, \quad \; \sup_{\varepsilon} \|  \pi_\varepsilon \| \le 1.
\end{equation}

\bigskip


\begin{theorem}\label{T1}
Let $T(t)$ be a strongly continuous, positive, contraction semigroup on $ C_0 (E)$ with generator $L$ defined by \eqref{LM}--\eqref{mu_prior}, and $T_\varepsilon$ be the linear operator on $l_0^{\infty}(E_\varepsilon)$ defined by \eqref{Te}.

Then for every $F \in C_0(E)  $
\begin{equation}\label{M-astral}
T_\varepsilon^{\left[ \frac{t}{\varepsilon^2} \right]} \pi_\varepsilon F \ \to \  T(t) F   \quad \mbox{for all} \quad t \ge 0
\end{equation}
as $\varepsilon \to 0$.
\end{theorem}

\begin{proof}
In view of (\ref{normle}) to prove (\ref{M-astral}) it suffices to show that
\begin{equation}\label{R1}
\|  T_\varepsilon^{\left[ \frac{t}{\varepsilon^2} \right]} \pi_\varepsilon F   - \pi_\varepsilon \ T(t) F  \|_{l_0^{\infty}(E_\varepsilon)} = \sup_{x \in \varepsilon \mathbb Z^d}\left| T_\varepsilon^{\left[ \frac{t}{\varepsilon^2} \right]} \pi_\varepsilon F (x, k(x))  - \pi_\varepsilon \ T(t) F (x, k(x)) \right| \to 0 \quad \mbox{as} \; \varepsilon \to 0.
\end{equation}
The proof of  (\ref{R1}) relies on the following  approximation theorem
\cite[Theorem 6.5, Ch.1]{EK}.\\[3mm]
{\bf Theorem}\,\cite{EK}.
{\sl \
For $n=1,2,\ldots$, let $T_n$ be a linear contraction on the Banach space ${\cal L}_n$, let $\varepsilon_n$ be a positive number, and put $A_n = \varepsilon_n^{-1}(T_n - E)$. Assume that $\lim_{n \to \infty} \varepsilon_n = 0$. Let $\{T(t) \}$ be a strongly continuous contraction semigroup on the Banach space ${\cal L}$ with generator $A$, and let $D$ be a core for $A$. Assume that $\pi_n: {\cal L} \to {\cal L}_n$ are bounded linear transformations with $\sup_n \|\pi_n \| < \infty$. Then the following are equivalent: \\

a) For each $f \in {\cal L} $, $T_n^{\left[ \frac{t}{\varepsilon_n} \right]} \pi_n f \ \to \  T(t) f \; $ for all  $t \ge 0$ as $\varepsilon \to 0$. \\

b) For each $f \in D$, there exists $f_n \in {\cal L}_n$ for each $n \ge 1$ such that $f_n \to f$ and $A_n f_n \to Af$.
}

\bigskip
According to this theorem  the semigroups convergence stated in item a) is equivalent to
the statement in item b) which is the subject of the  next lemma.

\begin{lemma}\label{Fn}
Let the operator $L$ be defined by \eqref{LM} -- \eqref{mu_prior} and the core $D \subset C_0(E)$ of $L$ be defined by \eqref{core}; assume that a bounded linear transformation $\pi_\varepsilon :\ C_0(E) \to l_0^{\infty}(E_\varepsilon)$ is defined by \eqref{1}, and $L_\varepsilon = \frac{1}{\varepsilon^2} ( T_\varepsilon -I)$.  Then for every $F \in D$, there exists $F_\varepsilon \in  l_0^{\infty}(E_\varepsilon)$ such that
\begin{equation}\label{F0}
 \| F_\varepsilon - \pi_\varepsilon F\|_{ l_0^{\infty}(E_\varepsilon)} \to 0
\end{equation}
and
\begin{equation}\label{F1}
\|L_\varepsilon F_\varepsilon - \pi_\varepsilon LF\|_{ l_0^{\infty}(E_\varepsilon )} \to 0 \quad \mbox{as } \; \varepsilon \to 0.
\end{equation}
\end{lemma}

\begin{proof}
For any $F = (f_0, f_1, \ldots, f_M) \in D$ we consider the following $F_\varepsilon \in l_0^{\infty}(E_\varepsilon)$
\begin{equation}\label{4}
F_\varepsilon (x, k(x)) \ = \ \left\{
\begin{array}{ll}
f_0 (x) + \varepsilon (\nabla f_0 (x), h(\frac{x}{\varepsilon})) +  \varepsilon^2 (\nabla \nabla f_0 (x), g(\frac{x}{\varepsilon}))  \\[3mm] +\varepsilon^2  \sum_{j=1}^M q_j (\frac{x}{\varepsilon}) (f_0(x) - f_j (x)),  & \mbox{if} \;  x \in \varepsilon B^{\sharp}, \; k(x) = 0, \\ \\
f_1 (x),  & \mbox{if} \;  x \in \varepsilon \{ x_1 \}^{\sharp}, \; k(x) = 1,\\ \cdots \\
f_M (x),  & \mbox{if} \;  x \in \varepsilon \{ x_M \}^{\sharp}, \; k(x) =M.
\end{array}
\right.
\end{equation}
Here $ h(y),  g(y), q_j (y), j = 1, \ldots, M,$ are periodic bounded functions  defined below. From (\ref{1}) and (\ref{4}) it immediately follows that
$$
\sup_{x \in \varepsilon \mathbb Z^d} |  F_\varepsilon (x,k(x)) - \pi_\varepsilon F(x, k(x))  | = \| F_\varepsilon - \pi_\varepsilon F\|_{ l_0^{\infty}(E_\varepsilon)} \to 0
$$ as $\varepsilon \to 0$. Thus convergence (\ref{F0}) is valid.
\medskip


In compliance with decomposition (\ref{PV}) for the transition matrix $P$ we introduce the operators:
\begin{equation}\label{PVeps}
T_\varepsilon = T_\varepsilon^0 + \varepsilon^2 V_\varepsilon,
\end{equation}
where
$$
T_\varepsilon^0 f(x,k(x))=\sum\limits_{y\in\varepsilon\mathbb Z^d}p_0\Big(\frac x\varepsilon,
\frac y\varepsilon\Big)f(y,k(y)),\quad
V_\varepsilon f(x,k(x))=\sum\limits_{y\in\varepsilon\mathbb Z^d}v\Big(\frac x\varepsilon,
\frac y\varepsilon\Big)f(y,k(y)).
$$

Let us  note that due to the structure of the matrix $P^0$, the operator $T_\varepsilon^0$ has a block structure:  $T_\varepsilon^0 f(x,k(x))=f(x,k(x))$ for $x\in \varepsilon A^\sharp$, and $T_\varepsilon^0|_{x\in \varepsilon B^{\sharp}}$ is defined by the transition probabilities of the random walk on the perforated lattice $\varepsilon B^{\sharp} = \varepsilon \mathbb Z^d \setminus \varepsilon A^{\sharp}$.

According to (\ref{PVeps}) the operator  $L_\varepsilon = \frac{1}{\varepsilon^2} ( T_\varepsilon -I)$ can be written as
$$
L_\varepsilon = \frac{1}{\varepsilon^2} ( T_\varepsilon^0 + \varepsilon^2 V_\varepsilon - I ) = L_\varepsilon^0 + V_\varepsilon,
$$
where $L_\varepsilon^0 = \frac{1}{\varepsilon^2} ( T_\varepsilon^0 - I)$ is the generator of the random walk on the perforated lattice $\varepsilon B^{\sharp} = \varepsilon \mathbb Z^d \setminus \varepsilon A^{\sharp}$.

To prove that
\begin{equation}\label{convL}
 \|L_\varepsilon F_\varepsilon - \pi_\varepsilon LF\|_{ l_0^{\infty}(E_\varepsilon)} = \sup_{x \in \varepsilon \mathbb Z^d } |L_\varepsilon F_\varepsilon (x, k(x)) - \pi_\varepsilon LF (x, k(x))|   \to 0
\end{equation}
we consider separately the case when $x \in \varepsilon B^{\sharp}$, and $x \in \varepsilon A^{\sharp}$.
 Since the second component in $E_\varepsilon$ is a function of the first one, in the remaining part of the proof for brevity  write $F_\varepsilon(x)$ instead of $F_\varepsilon(x,k(x))$.

Let $x \in \varepsilon B^{\sharp}$, then the first component of $F_\varepsilon$ can be written as a sum
\begin{equation}\label{5}
F_\varepsilon (x) = F_\varepsilon^P(x) + F_\varepsilon^Q(x), \quad  x \in \varepsilon B^{\sharp},
\end{equation}
where
\begin{equation}\label{FP}
F_\varepsilon^P(x) = f_0 (x) + \varepsilon \left(\nabla f_0 (x), h(\frac{x}{\varepsilon})\right) +  \varepsilon^2 \left(\nabla \nabla f_0 (x), g(\frac{x}{\varepsilon})\right),
\end{equation}
\begin{equation}\label{FQ}
 F_\varepsilon^Q (x) = \varepsilon^2 \sum_{j=1}^M q_j (\frac{x}{\varepsilon}) (f_0 (x) - f_j (x)).
\end{equation}
Then
\begin{equation}\label{Ldec}
L_\varepsilon F_\varepsilon = (L_\varepsilon^0 + V_\varepsilon)F_\varepsilon = L_\varepsilon^0 (F_\varepsilon^P  +  F_\varepsilon^Q) + V_\varepsilon F_\varepsilon =  L_\varepsilon^0 F_\varepsilon^P  + L_\varepsilon^0 F_\varepsilon^Q  + V_\varepsilon F_\varepsilon.
\end{equation}

\begin{proposition}\label{P1}
There exist  bounded periodic functions $h(y)=\{h_i(y)\}_{i=1}^d$ and $g(y)=\{g_{im}(y)\}_{i,m=1}^d$ (correctors) and a positive definite matrix $\Theta>0$, such that
\begin{equation}\label{PP1}
L_\varepsilon^0 F_\varepsilon^P  \ \to \Theta \cdot \nabla \nabla f_0, \quad \; \mbox{i.e. } \ \sup_{x \in \varepsilon B^{\sharp}}|L_\varepsilon^0 F_\varepsilon^P (x) -  \Theta \cdot \nabla \nabla f_0(x)| \to 0  \quad \mbox{as } \; \varepsilon \to 0,
\end{equation}
where $F_\varepsilon^P$ is defined in \eqref{FP}.
\end{proposition}

The proof of this proposition is based on the corrector techniques, it is given in the Appendix.
\bigskip

Using (\ref{5}) - (\ref{FQ}) and the continuity of the functions $f_j$ we have
\begin{equation}\label{8}
(L_\varepsilon^0 F_\varepsilon^Q + V_\varepsilon F_\varepsilon)(x) = \sum_{j=1}^M\Big((T^0_\varepsilon - I) q_j (\frac{x}{\varepsilon})\Big) (f_0(x) - f_j(x)) +  \sum_{j=1}^M v_\varepsilon(x, x_j) (f_j(x) - f_0(x)) + o(1),
\end{equation}
where $v_\varepsilon(x,x_j)=v\big(\frac x\varepsilon,\frac {x_j}\varepsilon \big)$, and  $o(1)$ tends to 0 as $\varepsilon\to0$.

\begin{proposition}\label{P2}
There exist bounded periodic functions $q_j(y)$  in the decomposition \eqref{8} and positive constants $\alpha_{0j}>0, \; j = 1, \ldots, M$, such that
\begin{equation}\label{7}
\sup_{x \in \varepsilon B^{\sharp}} \left| (L_\varepsilon^0 F_\varepsilon^Q + V_\varepsilon F_\varepsilon)(x) - \sum_{j=1}^M \alpha_{0j} (f_j(x) - f_0(x))  \right| \ \to \ 0  \quad \mbox{as } \; \varepsilon \to 0.
\end{equation}
\end{proposition}
\noindent
The proof of the proposition is given in the Appendix.
\bigskip

Combining (\ref{PP1}), (\ref{8}) and (\ref{7}) yields
\begin{equation}\label{B}
\sup_{x \in \varepsilon B^{\sharp} } |L_\varepsilon F_\varepsilon (x) - \pi_\varepsilon LF (x)|   \to 0 \quad \varepsilon \to 0.
\end{equation}

\bigskip

The next step is to prove that
\begin{equation}\label{10}
\sup_{x \in \varepsilon A^{\sharp} } |L_\varepsilon F_\varepsilon (x) - \pi_\varepsilon L F (x)|   \to 0 \quad \varepsilon \to 0.
\end{equation}
Let $x \in \varepsilon \{ x_k \}^{\sharp} \subset \varepsilon A^{\sharp}$. From (\ref{4}) and continuity of functions $f_k$ it follows that
$$
(L_\varepsilon F_\varepsilon)(x)= (L_\varepsilon^0 + V_\varepsilon) F_\varepsilon (x) = V_\varepsilon F_\varepsilon (x)
$$
\begin{equation}\label{11}
=\sum_{{j=1}\atop{j \neq k}}^M v(y_k, y_j) (f_j(x) - f_k(x)) +  \sum_{y \in B} v(y_k, y) (f_0(x) - f_k(x))+o(1) \quad \mbox{as } \varepsilon \to 0,
\end{equation}
where we have used the fact that $f_k (x') = f_k(x) + o(1)$ when $|x-x'| \to 0$.
Here $y \in \mathbb T^d$ are variables on the periodicity cell, $y_j\in\mathbb Z^d$, and $v(y_k, y_j)$ are the elements of the matrix $V$ given by (\ref{PV}). Thus if for every $j,k = 1, \ldots, M,\; j \neq k$, we set:
$$
\alpha_{kj} = v(y_k, y_j), \quad \alpha_{k0} = \sum_{y \in B} v(y_k, y), \quad  \lambda(k) = \sum_{{j=0}\atop {j \neq k}}^M \alpha_{kj}, \quad \mu_{kj} = \frac{\alpha_{kj}}{\lambda(k)},
$$
then relation (\ref{11}) implies (\ref{10}).

Finally, (\ref{convL}) is a consequence of (\ref{B}) and (\ref{10}), and Lemma \ref{Fn} is proved.
\end{proof}

It remains to recall that (\ref{R1}) is a straightforward consequence of the above approximation theorem. This completes the proof of Theorem \ref{T1}.

\end{proof}

\begin{remark}\label{R2}
In the next section we show that there exists a Markov process $\mathcal X(t)$ corresponding to the Feller semigroup $T(t)$. From Theorem \ref{T1} one can easily derive the convergence of finite dimensional distributions of the processes $X_\varepsilon(t)$ (random walks on $E_\varepsilon$ defined by \eqref{Te}) to those of $\mathcal X(t)$.
\end{remark}

\section{Invariance principle, convergence of the processes}

For the original process  $X_\varepsilon (t) = (\widehat X_\varepsilon (t),  k(\widehat X_\varepsilon (t)) ) $, which is the random walk on $E_\varepsilon$ (see (\ref{Te})), the second component $ k(\widehat X_\varepsilon (t)) \in \{0 ,1,\ldots, M \}$ is the function of the first component  $\widehat X_\varepsilon (t) \in \varepsilon \mathbb Z^d$ (see (\ref{Peps})). Thus Markov processes $\widehat X_\varepsilon (t)$ and $X_\varepsilon (t)$ are equivalent, i.e. the trajectories of $\{ \widehat X_\varepsilon(t) \}$ are isomorphic to trajectories of  $\{ X_\varepsilon(t) \}$. However, the second component of  $X_\varepsilon(t)$ plays the crucial role when passing to the limit $\varepsilon \to 0$. As has been shown in Section 2 the limit process $\mathcal X(t)$ preserves the Markov property only in the presence of the second component $k(t) \in  \{0,1, \ldots, M \}$, and this is an interesting asymptotic property of the processes  $ X_\varepsilon(t)$. It should be noted that in the process $\mathcal X(t)$ the second component is not a function of the first one anymore. This can be observed, in particular, from the structure of the limit generator, see (\ref{LM}) - (\ref{LA}).

In the previous section we justified the convergence of the semigroups, and consequently, the finite dimensional distributions of  $ X_\varepsilon(t)$. The goal of this section is to prove the existence of the limit process $\mathcal X(t)$ in $E$ with sample paths in $D_E [0, \infty)$ and to establish the invariance principle for the processes $X_\varepsilon (t)$. Namely, we show that $X_\varepsilon (t)$ converges in distributions to $\mathcal X(t)$ as $\varepsilon \to 0$ in the Skorokhod topology of $D_E [0, \infty)$.

\begin{theorem}\label{T2}
For any initial distribution $\nu \in {\cal P}(E)$ there exists a Markov process $\mathcal X(t)$ corresponding to the semigroup $T(t): C_0(E) \to C_0(E)$ with generator $L$ defined by \eqref{LM} -- \eqref{mu_prior} and with sample paths in $D_E [0, \infty)$.

If $\nu$ is the limit law of $X_\varepsilon(0)$, then 
\begin{equation}\label{TT2}
X_\varepsilon (t) \ \Rightarrow \ \mathcal X(t) \quad \mbox{ in } \; D_E [0, \infty) \; \mbox{ as } \; \varepsilon \to 0.
\end{equation}
\end{theorem}

\begin{proof}
The main idea of the proof is to combine the convergence of the finite dimensional distributions of $X_\varepsilon(t)$ (that is a consequence of Theorem \ref{T1} see Remark \ref{R2}) and the tightness of $X_\varepsilon(t)$ in $D_E [0, \infty)$.

We apply here Theorem 2.12 from \cite{EK}, Chapter 4. For the reader convenience we formulate it here.

\bigskip\noindent
{\bf Theorem} \cite{EK}.{\sl \ 
Let $E, E_1, E_2, \ldots $ be metric spaces with $E$ locally compact and separable. For $n=1,2, \ldots$ let $\eta_n:E_n \to E$ be measurable, let $\mu_n(x, \Gamma)$ be a transition function on $E_n \times {\cal B} (E_n)$, and suppose $\{ Y_n(k), \ k=0,1,2, \ldots \}$ is a Markov chain in $E_n$ corresponding to $\mu_n (x, \Gamma)$. Let $\epsilon_n >0$ satisfy $\lim_{n \to \infty} \epsilon_n = 0$. Define
$X_n(t) = \eta_n ( Y_n ([t/\epsilon_n]))$,
$$
T_n f(x) = \int f(y) \mu_n (x , dy), \quad f \in B(E_n),
$$
and $\pi_n: B(E) \to  B (E_n)$ by $\pi_n f = f\circ \eta_n$. Suppose that $\{T(t)\}$ is a Feller semigroup on $C_0 (E)$ and that for each $f \in C_0(E)$ and $t \ge 0$
\begin{equation}\label{EK1}
\lim_{n \to \infty} T_n^{[t/\epsilon_n]} f \ = \ T(t) f.
\end{equation}
If $\{ X_n (0) \}$ has limiting distribution $\nu \in {\cal P} (E)$, then there is a Markov process $X$ corresponding to $\{ T(t) \}$ with initial distribution $\nu$ and sample paths in $D_E [0, \infty)$, and $X_n \Rightarrow X$.
}

\bigskip
In our case, $E=\mathbb R^d \times \{0,1, \ldots, M\}$, $E_n=E_\varepsilon \subset E, \; \varepsilon = \frac1n$, and $\eta_n = \eta_\varepsilon: \ E_\varepsilon \to E$ is the measurable mapping for every $\varepsilon$, it is embedding of the set $E_\varepsilon$, isomorphic to the lattice $\varepsilon \mathbb Z^d$, to $E$.
The Markov chain $Y_n(m), m = 0,1, \ldots$ is the same as the random walk $X_\varepsilon (m) = (\widehat X_\varepsilon (m),  k(\widehat X_\varepsilon (m)) )$ on $E_\varepsilon$ (see (\ref{Te})) with the transition function $\mu_n ((x, k(x)), \ (y, k(y))) = p_\varepsilon (x,y)$.
The semigroup $T(t)$ on $C_0(E)$ generated by the operator $L$, see \eqref{LM} -- \eqref{mu_prior}, is the Feller semigroup by the Hille-Yosida theorem as was mentioned in the beginning of Section 2.
Setting $\epsilon_n  = \frac{1}{n^2}$ in (\ref{EK1}) we see that the convergence in (\ref{M-astral}) ensures the convergence in (\ref{EK1}).

Thus, all assumptions of Theorem 2.12 from \cite{EK} are fulfilled. Consequently, if we set $X_\varepsilon (t)= Y_n (\left[ \frac{t}{\varepsilon^2} \right])=\eta_n\big(Y_n (\left[ \frac{t}{\varepsilon^2} \right])\big)$, then  these processes convergence in law in the space $D_E[0, \infty)$. Theorem \ref{T2} is completely proved.
\end{proof}

\section{Generalization. Several fast components} 

In the final part of the paper we consider some generalizations of the  model studied above.
We keep all the assumptions on $p(x,y)$, in particular we assume that these transition probabilities are periodic, have a finite range of interaction and define an irreducible random walk, and that \eqref{PV} holds. We also keep
all the assumptions on $p_0(x,y)$ except for that on the structure
of the set $B^{\sharp}$. Here we assume that $B^{\sharp}$ is the union of $N$, $N> 1$, non-intersecting unbounded
periodic sets such that $P^0$ is invariant and irreducible on each of these sets.

We denote these sets $B^{\sharp}_1,\ldots, B^{\sharp}_N$ and assume that each $B^{\sharp}_j$, $j=1,\ldots, N$, is connected with respect to $P^0$ and, moreover, is a maximal connected component.
Our assumptions on the matrix $P^0$ now take the form
\begin{itemize}
\item[\bf --] $P^0$ satisfies conditions (\ref{p})
\item[\bf --] $p_0(x,x) = 1$, if $x \in A^\sharp$;
%
\item[\bf --] $p_0(x,y) = 0$, if $x, y \in A^\sharp, \ x \neq y$;
\item[\bf --] $p_0(x,y) = 0$, if $x \in A^\sharp, \ y \in B^\sharp$;
\item[\bf --] $p_0(x,y) = 0$, if $x \in B_i^\sharp, \ y \in B_j^\sharp$,\ $i,\,j=1,\ldots, N$,\ $i\not=j$.
%
%
\end{itemize}


As in Section \ref{sec2} we introduce the extended process on $\varepsilon \mathbb Z^d, \; \varepsilon \in (0,1)$. 
For each $k=1, \ldots, M$ we denote by $\{x_k \}^{\sharp}$ the periodic extension of the point $x_k \in A$, then
\begin{equation}\label{decZ_bis}
\varepsilon \mathbb Z^d  = \varepsilon B^{\sharp} \cup \varepsilon A^{\sharp} = \varepsilon B_1^{\sharp}\cup\ldots\cup \varepsilon B_N^\sharp \cup \varepsilon \{ x_1 \}^{\sharp} \cup \ldots  \cup \varepsilon \{ x_M \}^{\sharp}.
\end{equation}
We assign to each $x \in \varepsilon \mathbb Z^d$ the index $k(x) \in \{1, \ldots, N+M\}$ depending on the component in decomposition (\ref{decZ_bis}) to which $x$ belongs:
\begin{equation}\label{kx_bis}
k(x) \ = \ \left\{
\begin{array}{l}
j, \; \mbox{ if } \;  x \in \varepsilon B_j^{\sharp},\ j=1,\ldots,N; \\
N+j, \;  \mbox{ if } \;  x \in \varepsilon \{ x_j \}^{\sharp}, \; j= 1, \ldots, M.
\end{array}
\right.
\end{equation}
With this construction in hands we  introduce the metric space
\begin{equation}\label{Ee_bis}
E_\varepsilon\ = \  \left\{ (x, k(x)), \; x \in \varepsilon \mathbb Z^d,\; k(x) \in \{1,\ldots, N+M \} \right\}, \quad E_\varepsilon \subset \varepsilon \mathbb Z^d \times \{1,\ldots, N+M \}
 \end{equation}
 with a metric that coincides with the metric in  $\varepsilon \mathbb Z^d$ for the first component of $(x,k(x)) \in E_\varepsilon$.
As in Section \ref{sec2} we denote by ${\cal B}(E_\varepsilon)$ the space of bounded functions on $E_\varepsilon$ and introduce the transition operator $T_\varepsilon$ of random walk $X_\varepsilon(t) = (\widehat X_\varepsilon (t), k(\widehat X_\varepsilon (t)))$ on $E_\varepsilon$ as follows:
\begin{equation}\label{Te_bis}
(T_\varepsilon f)(x, k(x)) = \sum_{y \in \varepsilon \mathbb Z^d } p_\varepsilon (x,y) f(y ,k(y)), \quad f \in {\cal B}(E_\varepsilon).
\end{equation}
Then $T_\varepsilon$ is the contraction on  ${\cal B}(E_\varepsilon)$.
\bigskip

To construct the limit semigroup we denote $E= \mathbb R^d \times \{1, \ldots, N+M\}$, and $ C_0 (E)$ stands for the Banach space of continuous  functions vanishing at infinity. A function $F=F(x,k) \in C_0(E)$ can be represented as a vector function
$$
F(x,k) = \{ f_k (x) \in C_0(\mathbb R^d), \; k=1,\ldots, N+M \}.
$$
We introduce the operator
\begin{equation}\label{LM_bis}
(L F)(x,k) = \left(
\begin{array}{c}
\Theta^1 \cdot \nabla \nabla f_1 (x) \\
\cdots\\
\Theta^N \cdot \nabla \nabla f_N (x) \\
0  \\ \cdots \\ 0
\end{array}
\right) \ + \ L_A F(x,k),
\end{equation}
where  $\Theta^1,\ldots,\Theta^N$  are  positive definite matrices defined below in formula \eqref{defth_bis}, and  $L_A$ is a generator of the following Markov jump process
\begin{equation}\label{LA_bis}
L_A F(x,k) \ = \ \lambda(k)\ \sum_{{j=1}\atop{j \neq k}}^{N+M} \mu_{kj} (f_j(x) - f_k(x)).
\end{equation}
Here the parameters $\lambda(k)$ and $\mu_{kj}$ are determined as follows: first we define
transition intensities 
$$
\alpha_{kj}=\left\{
\begin{array}{ll}
\displaystyle
\frac1{|B_k|}\sum_{y\in B_k}\sum_{y'\in B_j}v(y,y'),\ \ &\hbox{if }k,j=1,\ldots, N,\ k\not=j; \\[3mm]
\displaystyle
\frac1{|B_k|}\sum_{y\in B_k}v(y,x_{j-N}),\ \ &\hbox{if }k=1,\ldots, N,\ j=N+1,\ldots,N+M; \\[3mm]
\displaystyle
\sum_{y\in B_j}v(x_{k-N},y),\ \ &\hbox{if }k=N+1,\ldots,N+M,\ j=1,\ldots, N; \\[4mm]
\displaystyle
v(x_{k-N},x_{j-N}),\ \ &\hbox{if }k,\,j=N+1,\ldots,N+M, \ k\not=j.
\end{array}
\right.
$$
and then set for each $k,\,j=1,\ldots, N+M$
\begin{equation}\label{deflamu_bis}
\lambda(k)=\sum\limits_{i=1}^{N+M}\alpha_{ki},\qquad \mu_{kj}=\frac{\alpha_{kj}}{\lambda(k)}. 
\end{equation}

The operator $L$ is defined on the core
\begin{equation}\label{core_bis}
D =  \{ (f_1, \ldots, f_{N+M}), \; f_j \in C_0^{\infty}(\mathbb R^d)\ \hbox{for }j=1,\ldots, N; \; f_j \in C_0(\mathbb R^d), \ \hbox{for } j=N+1, \ldots, N+M \},
\end{equation}
which is a dense set in $ C_0 (E)$.
As in Section \ref{sec2} one can check that the operator $L$ on $C_0 (E)$ satisfies the positive maximum principle.
Then by the Hille-Yosida theorem the closure of $L$ is a generator of a strongly continuous, positive, contraction semigroup $T(t)$  on $ C_0 (E)$.

\medskip
In this framework the operator  $\pi_\varepsilon:C_0(E)\mapsto l_0^{\infty}(E_\varepsilon)$ is defined as follows:
\begin{equation}\label{1_bis}
(\pi_\varepsilon  F) (x, k(x)) \ = \ \left\{
\begin{array}{ll}
f_1 (x), & \mbox{if} \; x \in \varepsilon B_1^{\sharp}, \; k(x) =1; \\
\cdots\\
f_N (x), & \mbox{if} \; x \in \varepsilon B_N^{\sharp}, \; k(x) =N; \\[2mm]
f_{N+1} (x), & \mbox{if} \; x \in \varepsilon \{ x_1 \}^{\sharp}, \; k(x) =N+1;\\
\cdots \\
f_{N+M} (x), & \mbox{if} \; x \in \varepsilon \{ x_M \}^{\sharp}, \; k(x) =N+M.
\end{array}
\right.
\end{equation}
It remains to define matrices $\Theta^j$ that appeared in \eqref{LM_bis}. In fact, for each $j=1,\ldots,N$, the matrix $\Theta^j$ coincides with  the effective diffusion matrix of the random walk on $B_j^\sharp$ with transition matrix $P^0$. We denote the restriction of $P^0$ on $B_j$ by $P^0_j$ and recall of the definition of the effective diffusion matrix.  To this end we consider, for each $j=1,\ldots, N$,  the equation
$$
 \sum_{\xi\in\Lambda_y} p_\xi (y) \left(  \xi + ( h^j(y+\xi) - h^j(y))\right)=0,\qquad y\in B_j^\sharp.
$$
As was shown in the proof of Proposition \ref{P1} in Appendix, this equation has a periodic solution which is unique
up to an additive constant. We set
\begin{equation}\label{defth_bis}
\Theta^j = \frac{1}{|B_j|} \sum_{y \in B_j} \sum_{\xi \in \Lambda_y} p_\xi (y)\, \xi \otimes \left( \frac12  \xi  +  h^j(y + \xi) \right).
\end{equation}

\begin{theorem}\label{T3_bis}
Let $T(t)$ be a strongly continuous, positive, contraction semigroup on $ C_0 (E)$ with generator $L$ defined by \eqref{LM_bis} -- \eqref{deflamu_bis}, \eqref{defth_bis}, and $T_\varepsilon$ be the linear operator on $l_0^{\infty}(E_\varepsilon)$ defined by \eqref{Te_bis}.

Then for every $F \in C_0(E)  $
\begin{equation*}\label{M-astral_bis}
T_\varepsilon^{\left[ \frac{t}{\varepsilon^2} \right]} \pi_\varepsilon F \ \to \  T(t) F   \quad \mbox{for all} \quad t \ge 0
\end{equation*}
as $\varepsilon \to 0$.

For any initial distribution $\nu \in {\cal P}(E)$ there exists a Markov process $\mathcal{X}(t)$ corresponding to the semigroup $T(t): C_0(E) \to C_0(E)$ with generator $L$ 
and sample paths in $D_E [0, \infty)$. Moreover, if the initial distributions $\nu_\eps\in {\cal P}(E_\varepsilon)$ of the processes $X_\varepsilon$ converge weakly, as $\varepsilon\to0$, to $\nu \in {\cal P}(E)$, then
\begin{equation*}\label{TT2_bis}
X_\varepsilon (t) \ \Rightarrow \ \mathcal{X}(t) \quad \mbox{ in } \; D_E [0, \infty) \; \mbox{ as } \; \varepsilon \to 0.
\end{equation*}
\end{theorem}

\begin{proof}
The proof of this Theorem follows the line of the proof of Theorems \ref{T1} and \ref{T2}. We leave it to the reader.
\end{proof}

The limit process $\mathcal X(t)$ can be described in the following way. Its second component is a Markov jump process with $N+M$ states whose intensities  and transition probabilities are
given in \eqref{deflamu_bis}. The first component that evolves in $\mathbb R^d$ remains still when  the second one takes on values in
$\{N+1, N+M\}$, and it shows a diffusive behaviour with the covariance $\Theta^j$ when the  second
component is equal to $j$, $j=1,\ldots, N$.

\section*{Appendix: proofs of the propositions}

\begin{proof}[Proof of Proposition \ref{P2}]

From (\ref{8}) we obtain the following system of uncoupled equations on the functions $q_j (\frac{x}{\varepsilon})$ and constants $\alpha_{0 j}$:
$$
 \Big((T^0_\varepsilon - I) q_j (\frac{x}{\varepsilon})\Big) (f_0(x) - f_j(x)) +   v(\frac{x}{\varepsilon}, \frac{x_j}{\varepsilon}) (f_j(x) - f_0(x)) = \alpha_{0j}  (f_j(x) - f_0(x)), \quad j = 1, \ldots, M.
$$
Then $q_j (\frac{x}{\varepsilon})$ satisfies, for every $j = 1, \ldots, M$, the equation
\begin{equation}\label{9}
(T^0_\varepsilon - I) q_j (\frac{x}{\varepsilon}) =   v(\frac{x}{\varepsilon}, \frac{x_j}{\varepsilon}) - \alpha_{0j} {\bf 1}, \quad x \in \varepsilon B^{\sharp}, \; x_j \in \varepsilon \{x_j \}^{\sharp},
\end{equation}
which is of equivalent the following equation on  $ B^\sharp$:
\begin{equation}\label{P2_4}
(P^0 - I) q_j (y) = v (y, y_j) - \alpha_{0j} {\bf 1}, \quad  y \in B^\sharp, \; y_j \in A^\sharp,
\end{equation}
where ${\bf 1}(y) = 1 \; \forall y \in B^\sharp$, and $q_j$ is $Y$-periodic.
Using Fredholm' alternative we conclude that the equation (\ref{P2_4}) has a unique solution if
$$
v(y,y_j)  - \alpha_{0j} {\bf 1} \ \bot \mbox{ Ker } (P^0 - I)^\ast = \{ {\bf 1} \}.
$$
The last relation follows from the irreducibility of $P^0$ on $B^\sharp$.
This condition implies the unique choice of constants $\alpha_{0j}$
\begin{equation}\label{alpha}
\alpha_{0j} \ = \ \frac{1}{|B|} \ \sum_{y \in B} v(y,y_j)>0 \quad \mbox{ with } \; y_j \in A,
\end{equation}
where $|B|$ is the cardinality of the set $B$. Thus $\alpha_{0j}$ is defined by (\ref{alpha}) for every $j =1, \ldots, M$, and the equation (\ref{P2_4}) has a unique solution $q_j(y)$ that is a bounded periodic function on the set $B^\sharp$.
Proposition \ref{P2} is proved.
\end{proof}
\bigskip

\begin{proof}[Proof of Proposition \ref{P1}]

We say that $y \sim x, \; x,y \in \mathbb Z^d$, if $p_0 (x,y) \neq 0$. Let $\Lambda_x$ be a finite set  of $y\in \mathbb Z^d$ such that $y \sim x$.
We will use further the notation
$$
p_0 (x,y) = p_0 (x, x+ \xi) = p_\xi(x) \; \mbox{ for all } \; x \sim y, \; x,y \in \mathbb Z^d, \; \mbox{ with } \;  y =x+ \xi.
$$
Then
$$
\sum_{\xi \in \Lambda_x} p_\xi (x)=1,
$$
and
\begin{equation}\label{Pxi1}
(T_0^\varepsilon f)(x) \ = \ \sum_\xi p_\xi (\frac{x}{\varepsilon}) f(x+ \varepsilon \xi), \quad x \in \varepsilon B^{\sharp}.
\end{equation}

Using (\ref{FP}) we get for all  $x \in \varepsilon B^{\sharp}$:
\begin{equation}\label{P1_1}
L_\varepsilon^0 F_\varepsilon^P (x) =  \frac{1}{\varepsilon^2} (T_\varepsilon^0 - I) \left( f_0 (x) + \varepsilon \left(\nabla f_0 (x), h(\frac{x}{\varepsilon}) \right) \right) +  (T_\varepsilon^0 - I) \left(\nabla \nabla f_0 (x), g(\frac{x}{\varepsilon})\right).
\end{equation}
Then the vector function $h(\frac{x}{\varepsilon})$ is taken from the relation
\begin{equation}\label{P1_2}
\frac{1}{\varepsilon^2} (T_\varepsilon^0 - I) \left( f_0 (x) + \varepsilon \left(\nabla f_0 (x), h(\frac{x}{\varepsilon}) \right) \right) = O(1).
\end{equation}
Using (\ref{Pxi1}) we obtain that the left-hand side of (\ref{P1_2}) takes the form:
\begin{equation}\label{P1_3}
\frac{1}{\varepsilon^2} \sum_\xi p_\xi (\frac{x}{\varepsilon}) \left( f_0 (x+ \varepsilon \xi) - f_0(x) \right)  +
\frac{1}{\varepsilon} \sum_\xi p_\xi (\frac{x}{\varepsilon})
\left( \left(\nabla f_0 (x+\varepsilon \xi), h(\frac{x}{\varepsilon}+\xi) \right) - \left(\nabla f_0 (x), h(\frac{x}{\varepsilon}) \right)  \right)
\end{equation}
$$
= \frac{1}{\varepsilon} \sum_\xi p_\xi (\frac{x}{\varepsilon}) \left( \nabla f_0 (x), \xi \right)  +
\frac{1}{\varepsilon} \sum_\xi p_\xi (\frac{x}{\varepsilon})
\left( \nabla f_0 (x), h(\frac{x}{\varepsilon}+\xi) - h(\frac{x}{\varepsilon}) \right)  + O(1)
$$
$$
= \frac{1}{\varepsilon}   \left( \nabla f_0 (x), \sum_\xi p_\xi (\frac{x}{\varepsilon}) \left(  \xi + ( h(\frac{x}{\varepsilon}+\xi) - h(\frac{x}{\varepsilon}) \right) \right)  + O(1).
$$
Thus the vector function $h(x)$ is a solution of the equation
\begin{equation}\label{P1_4}
(P^0 - I) \left( l (x) + h(x) \right) =0, \quad x \in B,
\end{equation}
 where $l(x)=x$ is the linear function. The solvability condition for equation (\ref{P1_4}) reads
$$
((P^0 - I)l,  \mbox{ Ker } (P^0 - I)^\ast) = ((P^0 - I) l ,\ {\bf 1}) = \sum_{x \in B} \sum_\xi p_\xi(x) \xi = 0.
$$
Since $p_{\xi}(x) = p_{-\xi}(x+\xi)$, this condition holds true, which implies the existence of the unique, up to an additive constant,  solution $h(x)$ of equation (\ref{P1_4}).

\bigskip

We follow the similar reasoning to find an equation for the matrix function $g(x), \ x \in B$. Collecting in (\ref{P1_1}) all terms of the order $O(1)$ and using  relation (\ref{P1_4}) on the  function $h(x)$ we get:
$$
\frac{1}{\varepsilon^2} \sum_\xi p_\xi (\frac{x}{\varepsilon}) \left( f_0 (x+ \varepsilon \xi) - f_0(x) \right)  +
\frac{1}{\varepsilon} \sum_\xi p_\xi (\frac{x}{\varepsilon})
\left( \left(\nabla f_0 (x+\varepsilon \xi), h(\frac{x}{\varepsilon}+\xi) \right) - \left(\nabla f_0 (x), h(\frac{x}{\varepsilon}) \right)  \right)
$$
$$
+ \sum_\xi p_\xi (\frac{x}{\varepsilon}) \left( \left( \nabla \nabla f_0 (x + \varepsilon \xi), g(\frac{x}{\varepsilon} + \xi) \right) -
 \left( \nabla \nabla f_0 (x ), g(\frac{x}{\varepsilon} ) \right) \right) + O(\varepsilon)
$$
$$
= \frac{1}{\varepsilon}   \left( \nabla f_0 (x), \sum_\xi p_\xi (\frac{x}{\varepsilon}) \left(  \xi + ( h(\frac{x}{\varepsilon}+\xi) - h(\frac{x}{\varepsilon}) \right) \right)
$$
$$
+ \left( \nabla \nabla f_0 (x), \sum_\xi p_\xi (\frac{x}{\varepsilon}) \left(  \frac12 \xi \otimes \xi + \xi \otimes  h(\frac{x}{\varepsilon}+\xi)  + (g(\frac{x}{\varepsilon} + \xi) - g(\frac{x}{\varepsilon} ) ) \right) \right) + O(\varepsilon)
$$
\begin{equation}\label{P1_5}
= \left( \nabla \nabla f_0 (x), \sum_\xi p_\xi (\frac{x}{\varepsilon}) \left(  \frac12 \xi \otimes \xi + \xi \otimes  h(\frac{x}{\varepsilon}+\xi) \right)  + (P^0 -I) g(\frac{x}{\varepsilon})   \right) + O(\varepsilon).
\end{equation}
Let $\frac{x}{\varepsilon} = y \in B$, and denote by $\Phi(h)$ the following matrix function
\begin{equation}\label{P1_6}
\Phi(h)(y) =   \frac12 \sum_\xi p_\xi (y)\, \xi \otimes \xi + \sum_\xi p_\xi (y)\, \xi \otimes h(y+ \xi), \quad y \in B.
\end{equation}
In order to ensure the convergence in (\ref{PP1}) we should find a constant matrix $\Theta$ and a periodic matrix
function $g(y)$ such that
\begin{equation}\label{P1_7}
\Phi(h)_{km}(y) + (P^0 - I) g_{km}(y) = \Theta_{km},
\end{equation}
The solvability condition for (\ref{P1_7}) reads
$$
(-\Phi(h)_{km} + \Theta_{km}, \ Ker (P^0 - I)^\ast) = (-\Phi(h)_{km} + \Theta_{km} ,\ {\bf 1})  = 0,
$$
thus $\Theta$ is uniquely defined as follows:
$$
\Theta_{k m} = \frac{1}{|B|} \sum_{y \in B} \Phi_{k m}(h) (y),
$$
and $g(y)$ is a solution of equation \eqref{P1_7}. This solution is uniquely defined up to a constant matrix.

\begin{proposition}\label{S1}
The matrix $\Theta$ defined by
\begin{equation}\label{theta}
\Theta = \frac{1}{|B|} \sum_{y \in B} \Phi (h) (y), \quad \mbox{ where } \quad \Phi(h)(y) = \sum_{\xi = \xi(y)} p_\xi (y)\, \xi \otimes \left( \frac12  \xi  +  h(y + \xi) \right)
\end{equation}
is  positive definite, i.e. $(\Theta \eta, \eta)>0 \; \forall \eta \neq 0$.
\end{proposition}

\begin{proof}
Step 1.  We prove that
\begin{equation}\label{step1}
\sum_{\xi = \xi(y)} \partial_{-\xi} a(y) \partial_{\xi} g(y) = -2(P^0 - I)g(y),
\end{equation}
where we denote $\partial_{\xi} g(y) = g(y+\xi) - g(y)$ for every $\xi= \xi(y)$, and $ a(y) = a_{\xi \xi}(y) = p_{\xi} (y)$ is the diagonal matrix. Using
$$
a(y) \partial_{\xi} g(y) = p_\xi (y) (g (y+\xi) - g(y)) \quad \mbox{ and } \;  p_{\xi} (y- \xi) =  p_{-\xi} (y),
$$
we obtain (\ref{step1}):
$$
\sum_{\xi} \partial_{-\xi} a(y) \partial_\xi g (y) = \sum_{\xi} \left( p_\xi (y-\xi) (g (y) - g(y- \xi))  - p_\xi (y) (g (y+\xi) - g(y)) \right) =
$$
$$
\sum_{\xi} \left(  p_{-\xi} (y) (g (y) - g(y- \xi)) - p_{\xi} (y) (g (y+\xi) - g(y)) \right) =  -2(P^0 - I) g(y).
$$
\medskip

Step 2. From (\ref{P1_4}) and (\ref{step1}) it follows that
$$
\sum_{\xi} \partial_{-\xi} a(y) \partial_\xi (l+h)(y) = 0, \quad y \in B.
$$
Consequently, for all $\eta \in \mathbb R^d$ we get
\begin{equation}\label{step2}
0 = \left( \sum_{y \in B} h(y) \sum_{\xi\in \Lambda_y} \partial_{-\xi} a(y) \partial_\xi (l+h)(y) \eta, \eta \right) = \left( \sum_{y \in B} \sum_{\xi\in \Lambda_y}  \partial_{\xi} h(y) a(y) \partial_\xi (l+h)(y) \eta, \eta \right).
\end{equation}

\bigskip

Step 3. On the other hand we have a positive definite quadratic form
\begin{equation}\label{step3}
\left( \sum_{y \in B} \sum_{\xi\in \Lambda_y} \partial_{\xi} (l+h)(y) a(y) \partial_\xi (l+h)(y) \eta, \eta \right)>0 \quad \forall \eta \neq 0,
\end{equation}
since $a(y) = \{p_\xi (y)\}$ is the diagonal matrix. Using (\ref{step2}) and (\ref{step3}) we have
\begin{equation}\label{step3.1}
\left( \sum_{y \in B} \sum_{\xi\in \Lambda_y} \partial_{\xi} l(y) a(y) \partial_\xi (l+h)(y) \eta, \eta \right)
=\left(\sum_{y \in B} \sum_{\xi\in \Lambda_y} p_\xi (y)\, \xi\otimes\left( \xi + h(y+ \xi) - h(y) \right) \eta, \eta \right) >0.
\end{equation}
Let us observe that
$$
- \sum_{ y \in B} \sum_{\xi\in \Lambda_y} p_\xi(y)\, \xi\otimes h(y) = \sum_{y \in B}  \sum_{\xi\in \Lambda_y}   p_{-\xi} (z)  (-\xi)\otimes h(z- \xi) =
\sum_{z \in B}  \sum_{\xi\in \Lambda_z}  p_{\xi} (z)\,  \xi\otimes h(z+ \xi);
$$
here we set $z = y + \xi$ and use the identity $
p_\xi (y) = p_{-\xi}(z)$.
Finally, the expression in (\ref{step3.1}) can be written as
$$
\left(\sum_{y \in B} \sum_{\xi\in \Lambda_y}  p_\xi (y)\, \xi\otimes\left( \xi + 2 h(y+ \xi) \right) \eta, \eta \right) =
$$
\begin{equation}\label{step3.2}
2 \left(\sum_{y \in B} \sum_{\xi\in \Lambda_y} p_\xi (y)\, \xi \otimes  \left( \frac12 \xi + h(y+ \xi) \right) \eta, \eta \right) =
2 \left(\sum_{y \in B} \Phi(h)(y) \eta, \eta \right) >0.
\end{equation}

\end{proof}

This complete the proof of Proposition \ref{P1}.

\end{proof}

\end{document}